\documentclass{amsart}
\usepackage{youngtab,amsmath,amssymb,amscd}

\providecommand{\sgn}{\mathop{\rm sgn}\nolimits}

\newtheorem{Lemma}      {Lemma}[section]
\newtheorem{Theorem}    [Lemma]{Theorem}
\newtheorem{Corollary}  [Lemma]{Corollary}
\newtheorem{Proposition}[Lemma]{Proposition}

\begin{document}

\title[Carter-Payne homomorphisms and branching rules for endomorphisms]
{Carter-Payne homomorphisms and branching rules for endomorphism rings of Specht modules}

\author{Harald Ellers}
\address{Department of Mathematics\\
   Allegheny College\
   Meadville, PA 16335\\
   USA}
\email{hellers@allegheny.edu}
\author{John Murray}
\address{Department of Mathematics\\
   National University of Ireland Maynooth\\
   Co.~Kildare\\
   Ireland}
\email{John.Murray@maths.nuim.ie}
\subjclass{20C20, 20C30}
\date{January 30, 2009}
\keywords{one-box-shift homomorphism, Jucys-Murphy elements}

\begin{abstract}
Let n be a positive integer and let p be a prime. Suppose that we take a partition of n, and obtain another partition by moving a node from one row to a shorther row. Carter and Payne showed that if the p-residue of the removed and added positions is the same, then there is a non-zero homomorphism between the corresponding Specht modules for the symmetric group of degree n, defined over a field of characteristic p. In this paper we give a very simple description of such a homomorphism, as a map between polytabloids, using the action of a Murphy-Jucys element. We also present a proof that in this context the homomorphism space is 1-dimensional (this is a special case of the more general result for Iwahori-Hecke algebras proved by S. Lyle). Our methods give a lower bound on where the image of the Carter-Payne homomorphism lies in the Jantzen filtration of the codomain Specht module.

Let $\Sigma_n$ be the symmetric group of degree $n$, and let $F$ be a field of characteristic $p\ne 2$. Suppose that $\lambda$ is a partition of\/ $n+1$, that $\alpha$ and $\beta$ are partitions of\/ $n$ that can be obtained by removing a node of the same residue from $\lambda$, and that $\alpha$ dominates $\beta$. Let $S^\alpha$ and $S^\beta$ be the Specht modules, defined over $F$, corresponding to $\alpha$, respectively $\beta$. We give a very simple description of a non-zero homomorphism $\theta:S^\alpha\rightarrow S^\beta$ and present a combinatorial proof of the fact that dim$\,{\rm Hom}_{F\Sigma_n}(S^\alpha,S^\beta) = 1$. As an application, we describe completely the structure of the ring ${\rm End}_{F\Sigma_n}({S^\lambda}\,{\downarrow_{\Sigma_n}})$. Our methods furnish a lower bound for the Jantzen submodule of $S^\beta$ that contains the image of $\theta$.
\end{abstract}

\maketitle

\section{Introduction}

Let $n$ be a positive integer and let $\Sigma_n$ denote the symmetric group of degree $n$. For any ring $R$ and any partition $\alpha$ of\/ $n$, the Specht module $S_{R}^{\alpha}$ is defined to be the submodule of the permutation module $R_{\Sigma_{\alpha}}{\uparrow^{\Sigma_n}}$ spanned by certain elements called polytabloids, where $\Sigma_{\alpha}$ is the Young subgroup associated to $\alpha$ and $R_{\Sigma_{\alpha}}$ is the trivial $R\Sigma_{\alpha}$-module. (See \cite{James} for definitions.) Specht modules play a central role in the representation theory of the symmetric group. Now suppose that $F$ is a field. If $F$ has characteristic $0$, the Specht modules defined over $F$ are the simple $F\Sigma_n$-modules; if $F$ has characteristic $p$ the heads of the Specht modules $S_{F}^{\alpha}$, with $\alpha$ $p$-regular, are the simple $F\Sigma_n$-modules.

It would be useful to understand ${\rm Hom}_{F\Sigma_n}(S^{\alpha},S^{\beta})$  for all partitions $\alpha$ and $\beta$ of $n$ when $F$ has positive characteristic, but this problem has not been solved in general. In \cite{CarterPayne}, Carter and Payne exhibited some types of partitions $\alpha$ and $\beta$ for which ${\rm Hom}_{\Sigma_n}(S^{\alpha},S^{\beta}) \neq 0$ even though $\alpha \neq \beta$, generalizing earlier results of Carter and Lusztig \cite{CarterLusztig}.

Among the examples discovered by Carter and Payne are {\em one-box shift} homomorphisms. Suppose that the Young diagram $[\beta]$ of $\beta$ can be obtained by moving a single node in the diagram $[\alpha]$; then we say that $\beta$ is {\em one-box-shift} of $\alpha$ (see \eqref{E:beta}). This is equivalent to saying that there exists a partition $\lambda$ of $n+1$ such that $\alpha$ and $\beta$ can each be obtained by removing a node from the Young diagram of $[\lambda]$.  Carter and Payne showed that if the node in $[\alpha]\backslash[\beta]$ has the same $p$-residue as the node in $[\beta]\backslash[\alpha]$, and $\alpha$ dominates $\beta$, then  ${\rm Hom}_{\Sigma_n}(S^{\alpha},S^{\beta}) \neq 0$ (the {\em residue} of the node in the $i$-th row and $j$-th column of a Young diagram is the integer $j-i$; the {\em $p$-residue} is just $j-i$ taken modulo $p$).

In fact, in this case ${\rm dim}\,{\rm Hom}_{F\Sigma_n}(S^\alpha,S^\beta) = 1$, when the characteristic of $F$ is greater than $2$. This was proved in the context of Iwahori-Hecke algebras  by Lyle \cite{Lyle}.  For the group algebras, this result seems to have been known to experts in the area for some time. For example, A. Kleshchev claims, apparently incorrectly, in \cite[p498]{Kleshchev}, that this is a special case of the main result in \cite{CarterLusztig}. However, we do not believe that a complete proof of this theorem had been published before Lyle's paper.

Let $L_{n+1}$ be the Jucys-Murphy element consisting of the sum in $F\Sigma_{n+1}$ of all transpositions in $\Sigma_{n+1}$ that are not in $\Sigma_n$. We would like to point out that results from \cite{EllersMurray} make it possible to give the following simple description of one-box shift homomorphisms:

\begin{Theorem}\label{T:CarterPayne}
Let $F$ be a field of characteristic $p>0$ and let $\lambda$ be a partition of $n+1$. Suppose that $\alpha$ and $\beta$ can each be obtained by removing a node from $\lambda$, with $\alpha\unrhd\beta$, and that the removed nodes have the same $p$-residues. Identify $S_F^{\alpha}$ and $S_F^{\beta}$ as subquotients of $S_F^{\lambda}\downarrow_{\Sigma_n}$. Then multiplication by an appropriate polynomial in $L_{n+1}$ induces a non-zero homomorphisms from $S_F^{\alpha}$ to $S_F^{\beta}$ 
\begin{proof}
Assume that $\lambda$ has $t$ removable nodes, with $p$-residues $r_1,\ldots,r_t$. Let $\lambda_i$ be the partition of $n$ obtained by removing the $i$-th removable node from $[\lambda]$. Following James \cite[9.3]{James}, let $S^\lambda_i$ be the subspace of $S_F^\lambda$ spanned by those polytabloids $e_t$ where $t$ is a standard $\lambda$-tableau such that the symbol $n+1$ is in one of the first $i$ removable nodes of $t$. Then $S^\lambda_i$ is a $\Sigma_n$-submodule of $S^\lambda$, and $S^\lambda_i/S^\lambda_{i-1}\cong S^{\lambda_i}$. The $\Sigma_n$-series
$$
0\subset S^\lambda_0\subset S^\lambda_2\subset\ldots\subset S^\lambda_t=S^\lambda
$$
is called the Specht series of $S_F^\lambda{\downarrow_{\Sigma_n}}$.

Suppose that $\alpha=\lambda_u$ and that $\beta=\lambda_v$. Then $u<v$.  In \cite{EllersMurray}, we showed that $S_u^\lambda\prod_{i=u}^{v-1}(L_{n+1}-r_i)\subseteq S_v^\lambda$ and that this map induces a well-defined nonzero homomorphism $S_F^{\alpha} = S_u^\lambda/S_{u-1}^\lambda\rightarrow S_v^\lambda/S_{v-1}^\lambda = S_F^{\beta}$.

If $p>2$, then ${\rm Hom}_{\Sigma_n}(S_F^\alpha,S_F^\beta)$ is $1$-dimensional, and this map must be a nonzero multiple of the homomorphism defined in \cite{CarterPayne}.
\end{proof}
\end{Theorem}

In order to effectively use this result, we need some information about the action of $L_{n+1}$ on $S^\lambda$. The following is a variant of (3.3) in \cite{Murphy}:

\begin{Lemma}\label{L:MurphyMult}
Let $\lambda$ be a partition of $n+1$ and let $t$ be a $\lambda$-tableau (not necessarily standard). Suppose that $n+1$ occupies a column of length $r$ in $t$. Then $\lambda$ has one removable node $(r,c)$ for some $c\geq1$. Let $R$ be the set of symbols in $t$ whose columns have length strictly smaller that $r$. Then in $S^\lambda_{\mathbb Z}$ we have
$$
e_tL_{n+1}=(c-r)e_t+\sum_{i\in R}e_t(n+1,i).
$$
\end{Lemma}

For example, let $\lambda=[4,3,1]$, $n=7$ and $t=\tiny{\young(1234,567,8)}$. Then $n+1=8$ occupies a node of residue $-2$ in $t$. We use Lemma \ref{L:MurphyMult} to compute $e_t(L_8+2)$. For emphasis we replace the symbol $n+1$ by $\bullet$. The above lemma show that:
\begin{equation}\label{E:MurphyMult}
e_{\tiny\young(1234,567,\bullet)}\,\,(L_8+2)\,\,=\,\,
e_{\tiny\young(1\bullet 34,567,2)}+e_{\tiny\young(1234,5\bullet 7,6)} + e_{\tiny\young(12\bullet 4,567,3)}+
e_{\tiny\young(1234,56\bullet ,7)}+e_{\tiny\young(123\bullet,567,4)}\\
\end{equation}

We use this to illustrate Theorem \ref{T:CarterPayne}. Let $\alpha=[4,3]$ and $\beta=[3,3,1]$. For $p=5$ and $n=7$ there is a one-box shift $\theta:S^{\alpha}\rightarrow S^{\beta}$. Now $u=1,v=3$ and $r_1=3,r_2=1$ and $r_3=-2$. The relevant polynomial in $L_{n+1}$ is $(L_8+2)(L_8-1)$. We find the image of the polytabloid $e_t$ with $t=\tiny{\young(1234,567)}$ under $\theta$.  We again replace $8$ by $\bullet$. First note that $e_{\tiny\young(1234,567,\bullet)}$ is an element of $S^\lambda_3$ whose image is $e_{\tiny\young(1234,567)}$, under $S^\lambda_3/S^\lambda_2\cong S^\alpha$. We consider the summands in $e_t(L_8+2)$ given in \eqref{E:MurphyMult} in turn:
$$
\begin{array}{lcl}
e_{\tiny\young(1\bullet34,567,2)}\,\,(L_8-1)\,\,&=&\,\,e_{\tiny\young(143\bullet,567,2)}\\
e_{\tiny\young(1234,5\bullet7,6)}\,\,(L_8-1)\,\,&=&\,\,e_{\tiny\young(123\bullet,547,6)}\\
e_{\tiny\young(12\bullet4,567,3)}\,\,(L_8-1)\,\,&=&\,\,e_{\tiny\young(124\bullet,567,3)}\\
e_{\tiny\young(1234,56\bullet,7)}\,\,(L_8-1)\,\,&=&\,\,e_{\tiny\young(123\bullet,564,7)}\\
e_{\tiny\young(123\bullet,567,4)}\,\,(L_8-1)\,\,&=&\,\,e_{\tiny\young(123\bullet,567,4)}\,\,(3-1).
\end{array}
$$
Deleting $\bullet$ from each tableau, we get a tableau of shape $\beta$. So
\begin{equation}\label{Ex:OneBox}
e_{\tiny\young(1234,567)}\,\theta\,=\,
e_{\tiny\young(143,567,2)}+e_{\tiny\young(123,547,6)}+e_{\tiny\young(124,567,3)}+e_{\tiny\young(123,564,7)}+2\,e_{\tiny\young(123,567,4)}.
\end{equation}

Our second goal is to complete our analysis (see \cite{EllersMurray}) of the endomorphism ring of the restriction of the Specht module $S^\lambda$ to $\Sigma_n$.

\begin{Theorem}\label{T:corollary1}
Let $F$ be a field of characteristic $p$ not equal to $2$ and let $\lambda$ be a partition of $n+1$. Then the blocks of ${\rm End}_{F\Sigma_n}({S^\lambda}\,{\downarrow_{\Sigma_{n}}})$ are parametrized by the $p$-residues of the removable nodes in $[\lambda]$. Let $r$ denote the $p$-residue of a removable node of $\lambda$. Then the corresponding block of ${\rm End}_{F\Sigma_n}({S^\lambda}\,{\downarrow_{\Sigma_{n}}})$ has the structure of a truncated polynomial ring over $F$. Its dimension is equal to the number of removable nodes in $[\lambda]$ that have $p$-residue $r$.
\end  {Theorem}

The corresponding statement for the induced Specht module is also true:

\begin{Theorem}\label{T:corollary2}
Let $F$ be a field of characteristic $p$ not equal to $2$ and let $\lambda$ be a partition of $n-1$. Then the blocks of ${\rm End}_{F\Sigma_n}({S^\lambda}\,{\uparrow^{\Sigma_{n}}})$ are parametrized by the $p$-residues of the addable nodes in $[\lambda]$. Let $r$ denote the $p$-residue of an addable node of $\lambda$. Then the corresponding block of ${\rm End}_{F\Sigma_n}({S^\lambda}\,{\uparrow^{\Sigma_{n}}})$ has the structure of a truncated polynomial ring over $F$. Its dimension is equal to the number of addable nodes in $[\lambda]$ that have $p$-residue $r$.
\end  {Theorem}

It is interesting to compare these theorems to a result of Kleshchev, which says that if $D$ is a simple $F\Sigma_n$-module, then ${\rm End}_{F\Sigma_{n-1}}(D\,{\downarrow_{\Sigma_{n-1}}})$ is also a direct sum of truncated polynomial algebras, with one summand for each block of $F\Sigma_{n-1}$. (See Theorems 11.2.7 and 11.2.8  of \cite{Kleshchev2}.)

In later sections, we revisit the dimension of the space of one-box-shift homomorphisms.

\begin{Theorem}\label{T:main}
Let $F$ be a field of characteristic not equal to $2$ and let $\beta$ be a one-box-shift of $\alpha$. Then ${\rm dim}\,{\rm Hom}_{F\Sigma_n}(S^\alpha,S^\beta)\,\leq\,1$.
\end  {Theorem}

Our aim is to show that a proof is possible that uses only the relatively elementary combinatorial tools developed by G. James in his monograph \cite{James}. In particular, we do not use the Schur functor, nor do we make use of any of the deeper machinery of algebraic groups. We also appear to avoid some technical difficulties that arise in the context of Iwahori-Hecke algebras c.f. \cite{Lyle}.

All three theorems \ref{T:corollary1}, \ref{T:corollary2} and \ref{T:main} are false when $p=2$. The proof of Theorem \ref{T:main} relies on Lemma \ref{L:repeated}, which fails for $p=2$. In the case of Theorems \ref{T:corollary1}, and \ref{T:corollary2}, when $p=2$, we do not even know how to parametrize the blocks of ${\rm End}_{F\Sigma_n}({S^\lambda}\,{\downarrow_{\Sigma_{n}}})$ nor those of ${\rm End}_{F\Sigma_n}({S^\lambda}\,{\uparrow^{\Sigma_{n}}})$. For these reasons we now fix $F$ as a field of characteristic $p\ne 2$.

The methods we use to prove Theorem \ref{T:main} make it possible to explore how one-box-shift homomorphisms interact with Jantzen layers (see Section \ref{S:Jantzen} for definitions).

\begin{Theorem}\label{T:Jantzen}
Let $\lambda$ be a partition of $n+1$ and let $F$ be a field of characteristic $p>0$. Suppose that $\alpha$ and $\beta$ can be obtained by removing a node of residue $r_\alpha$ and $r_\beta$, respectively, from $\lambda$, with $\alpha\unrhd\beta$. Let $\theta:S_F^\alpha\rightarrow S_F^\beta$ be a $\Sigma_n$-homomorphism. If $p^i$ divides $r_\alpha-r_\beta$ then $\theta(S_F^\alpha)$ is contained in the $i$-th Jantzen submodule of $S_F^\beta$.
\end{Theorem}

\section{Notation and basic results}

For the rest of the paper, unless stated otherwise, $F$ is a field of characteristic $p>2$ and $\alpha$ is a fixed partition of $n$.  We also fix a bijection
$$
t:[\alpha]\rightarrow\{1,\ldots,n\}.
$$
The Young diagram $[\alpha]$ is a collection of boxes in the plane that is oriented left to right and top to bottom. This means that the first row is the one at the top and the first column is the one at the left. The $(i,j)$-th node of $[\alpha]$ is the box in the $i$-th row and the $j$-th column.

We refer to any map $T:[\alpha]\rightarrow\{1,\ldots,n\}$ as an {\em $\alpha$-tableau}. Informally, $T$ is a filling of the nodes of\/ $[\alpha]$, using elements from $\{1,\ldots,n\}$, with repeats allowed. The value of\/ $T$ at a node $b=(r,c)$ is denoted by $T_b$ or $T(r,c)$. If\/ $i=T_b$, we say that node $b$ of\/ $T$ {\em contains} the symbol $i$.

Identify $\Sigma_n$ with the group of permutations of\/ $\{1,\ldots,n\}$. Permutations and homomorphisms will generally act on the right. For instance $\Sigma_n$ acts on the set of all bijective $\alpha$-tableau by functional composition. We can use the bijection $t$ to give a contragradient action of\/ $\Sigma_n$ on all $\alpha$-tableau: if $\pi\in\Sigma_n$ then define the $\alpha$-tableau $T\pi$ as the
map $T\pi=t\circ\pi^{\!-1}\!\circ t^{\!-1}\!\circ T$.

Suppose that $\alpha$ has $r$ nonzero parts $[\alpha_1\geq\alpha_2\geq\ldots\geq\alpha_r]$. An {\em $\alpha$-tabloid} is an ordered partition of\/ ${\widehat n}$ whose parts have cardinalities $\alpha_1,\alpha_2,\ldots,\alpha_r$. For instance the sets $\{t_{ij}\mid j=1,\ldots,\alpha_i\}$ for $i=1,\ldots,r$, determine an $\alpha$-tabloid that we shall denote by $\{t\}$.

Let $R_t$ be the row stabilizer and let $C_t$ be the column stabilizer of\/ $t$. Denote by $M^{\alpha}$ the $F\Sigma_n$-module consisting of all formal $F$-linear combinations of\/ $\alpha$-tabloids. Then $\Sigma_n$ acts on $\alpha$-tabloids in the obvious way: $\{t\}\pi:=\{t\pi\}$. James shows that the corresponding $\Sigma_n$-module $M^\alpha$ is isomorphic to the permutation
module $F_{R_t}{\uparrow^{\Sigma_n}}$.

The following element of\/ $M^\alpha$ is called a {\em polytabloid}:
$$
e_t:=\sum_{\pi\in C_t}\sgn \pi \{t\pi\}.
$$
The subspace of $M^\alpha$ spanned by polytabloids is a $\Sigma_n$-module called the Specht module $S^\alpha$. For compactness we use $X^-:=\sum_{\sigma\in X}\sgn(\sigma)\sigma$, whenever $X\subseteq\Sigma_n$. So in this notation $e_t=\{t\}C_t^-$.

The following important important combinatorial result is one half of \cite[13.5]{James}.

\begin{Lemma}\label{L:repeated}
Suppose that $T$ is an $\alpha$-tableau that contains repeated entries in one of its columns. Then $T\,C_t^-=0$.
\begin{proof}
By hypothesis there exist integers $r,s,c$, with $r\ne s$, such that $T_{r,c}=T_{s,c}$. Then $\pi:=(t_{r,c},t_{s,c})$ is a transposition in $C_t$. So $\pi C_t^-=\sgn(\pi)\,C_t^-=-C_t^-$. But $T\pi=T$. It follows that $TC_t^-=T\pi C_t^-=-TC_t^-$. As ${\rm char}(F)\ne 2$, this implies that $TC_t^-=0$.
\end{proof}
\end{Lemma}

\section{Branching rules for endomorphism rings}

As in \cite{EllersMurray}, we use $E_n$ denote the sum, in $F\Sigma_n$, of all transpositions in $\Sigma_n$. Each $F\Sigma_n$-module $M$ comes equipped with a ring homomorphism $F\Sigma_n\rightarrow{\rm End}_F(M)$.

\begin{Theorem}\label{T:Ln-generates}
Let $\lambda$ be a partition of\/ $n$. Then the image of\/ $E_{n+1}$ generates\/ ${\rm End}_{F\Sigma_{n+1}}({S^\lambda}\,{\uparrow^{\Sigma_{n+1}}})$ and the image of\/ $E_{n-1}$ generates\/
${\rm End}_{F\Sigma_{n-1}}({S^\lambda}\,{\downarrow_{\Sigma_{n-1}}})$.
\begin{proof}
Let 
$M$ be 
the a summand of the restricted or induced module belonging to a single block of the group algebra.
We change notation from $n+1$ to $n$ in the first case, and from $n-1$ to $n$ in the second case, so that $M$ is an $F\Sigma_n$-module, and $E_n$ plays the role of $E_{n+1}$ or of $E_{n-1}$. Let $\epsilon$ denote the image of $E_n$ in ${\rm End}_{F}(M)$. Then \cite{James} and the authors' paper \cite{EllersMurray} shows that $M$ has the following properties:
\begin{enumerate}
\item $M$ has a Specht series
$$
0=M_0\subset M_1\subset\ldots\subset M_m=M,
$$
such that $M_{u}/M_{u-1}\cong S^{\lambda_u}$ where $\lambda_u$ is a partition of $n$. Moreover, $\lambda_v$ is a one-box-shift of $\lambda_u$, for each $v\ne u$.
\item The labelling partitions satisfy$\lambda_1\triangleleft\ldots\triangleleft\lambda_n$.
\item The endomorphism $\epsilon$ has minimal polynomial $(x-r)^m$. Here $\epsilon$ acts as a fixed scalar $r$ on each quotient module $M_u/M_{u-1}$, for $u=1,\ldots,m$.
\end  {enumerate}
It is known that the scalar $r$ is the sum, taken modulo $p$, of the residues of the nodes of the Young diagram of any partition
$\lambda_u$ that is associated with the given $p$-block of $\Sigma_n$. For convenience, we define $M_u:=0$, when $u\leq 0$.

Let $\theta$ be any $F\Sigma_n$-homomorphism of $M$. We shall prove by induction that given $i\geq 0$ there exists a polynomial $f_i\in F[x]$ such that $M_u(\theta-f_i(\epsilon))\subseteq M_{u-i}$, for $u=1,\ldots,m$. So $\theta=f_m(\epsilon)$ lies in the subalgebra generated over $F$ by $\epsilon$ and $1_M$, thus proving the Theorem.

As $p\ne 2$ and $\lambda_u\triangleleft\lambda_v$ for $u<v$, it follows from Corollary 13.17 of \cite{James} that there are
no non-zero homomorphisms $M_u/M_{u-1}\rightarrow M_v/M_{v-1}$, when $u<v$. This implies in particular that $M_u\theta\subseteq M_u$,
for each $u=1,\ldots,m$ (see Lemma 3.2 of \cite{EllersMurray} for details). This proves the base case for our induction.

Suppose then that $i>0$ is an integer such that there exists a polynomial $f_i\in F[x]$ such that $M_u(\theta-f_i(\epsilon))\subseteq M_{u-i}$, for $u=1,\ldots,m$. Then there exist well-defined $\Sigma_n$-homomorphisms
$\psi_{u,i}:M_u/M_{u-1}\rightarrow M_{u-i}/M_{u-i-1}$ given by $(v+M_{u-1})\psi_{u,i}=v(\theta-f_i(\epsilon))+M_{u-1}$,
for all $v\in M_u$.

Now $M_u(\epsilon-r)^i\subseteq M_{u-i}$, for all $u$. As $\epsilon$ has minimal polynomial of degree $m$, the induced maps $(\epsilon-r)^i_u:M_u/M_{u-1}\rightarrow M_{u-i}/M_{u-i}$ are non-zero, for $u>i$. But $\lambda_{u-i}$ is a one-box-shift of $\lambda_u$. So ${\rm Hom}_{F\Sigma_n}(\lambda_u,\lambda_{u-i})$ is a $1$-dimensional $F$-space, according to Corollary \ref{C:main}.
We deduce that $(\epsilon-r)^i_u$ is a basis for the space ${\rm Hom}_{F\Sigma_n}(\lambda_u,\lambda_{u-i})$, for all $u,i$.

The previous two paragraphs imply that there exist $\Lambda_u\in F$ such that $\psi_{u,i}=\Lambda_u(\epsilon-r)^i_u$, for $u=1,\ldots,m$. Then $M_u(\theta-f_{i+1}(\epsilon))\subseteq M_{u-i}$, for $u=1,\ldots,m$, where $f_{i+1}(x)=f_i(x)+\Lambda_m(x-r)^i$, if we can show that all $\Lambda_u$ are equal to $\Lambda_m$. Lemma 3.1 of \cite{EllersMurray} furnishes us with an element $\tau\in M$ such that $\tau_u:=\tau\prod_{i=u+1}^m(\epsilon-r)$ belongs to $M_u\backslash M_{u-1}$, for each $u\in\{1,\ldots,m\}$. Thus
$$
\begin{array}{lcl}
\Lambda_u\tau_{u-i}(\epsilon-r)^i
	&=&\tau_u(\theta-f_i(\epsilon)),
		\quad\mbox{mod $M_{u-i-1}$},\\
	&=&\tau(\epsilon-r)^{m-u}(\theta-f_i(\epsilon))\\
	&=&\tau(\theta-f_i(\epsilon))(\epsilon-r)^{m-u},
		\quad\mbox{as $\theta\in{\rm End}_{F\Sigma_n}(M)$},\\
	&\equiv&\Lambda_m\tau(\epsilon-r)^i(\epsilon-r)^{m-u},
		\quad\mbox{mod $M_{m-i-1}$},\\
	&\equiv&\Lambda_m\tau(\epsilon-r)^i,
		\quad\mbox{mod $M_{u-i-1}$}.
\end  {array}
$$
We conclude that $\Lambda_u=\Lambda_m$. This completes the proof of the inductive step, and hence of the Corollary.

Now consider the map from the polynomial algebra $F[x]$  to the summand 
associated to a single block of the endomorphism algebra that sends 
$x$ to multiplication by $E_n$. We have shown that this map is surjective and 
has kernel generated by a power of $x-r$, where $E_n$ acts as the scalar 
$r$ on each simple module in the block. If we consider a different block, the scalar $r$ 
will be different. The theorem follows.  
 
\end  {proof}
\end  {Theorem}

\begin{proof}[Proof of Theorem \ref{T:corollary1}]
Each $p$-block of $\Sigma_n$ has an associated positive integer $w$ known as its weight and a $p$-core $\kappa$.
Here $\kappa$ is a partition of $n-pw$ that has no hook-lengths divisible by $p$. A Specht module $S^\mu$ belongs to the block if and only if the $p$-core of $\mu$ coincides with $\kappa$. It is well known that the $p$-core of a partition is determined by the multi-set of residues (mod $p$) of the nodes in its Young diagram. Fix the residue $r$ of a removable node in $[\lambda]$.
Then all partitions that can be obtained by removing a single node of $p$-residue $r$ from $\lambda$ belong to the same $p$-block,
call it $b$, of $\Sigma_n$.

The main result of \cite{EllersMurray}, Theorem 3.4, shows that ${S^\lambda}\,{\downarrow_{\Sigma_n}}b$ is indecomposable.
As this is an ideal and a direct summand of ${\rm End}_{F\Sigma_n}({S^\lambda}\,{\downarrow_{\Sigma_n}})$, it is a block of this algebra. Theorem \ref{T:Ln-generates} implies that ${\rm End}_{F\Sigma_n}({S^\lambda}\,{\downarrow_{\Sigma_n}})b$ is generated by over $F$ by the image of $E_n$. Let $m$ be the number of removable nodes in $[\lambda]$ that have $p$-residue $r$. Then $E_n$ has minimal polynomial $(x-r)^m$. It follows that ${\rm End}_{F\Sigma_n}({S^\lambda}\,{\downarrow_{\Sigma_n}})b$ has the structure of the truncated polynomial ring $F[x]/(x^m)$.
\end{proof}

The proof of Theorem \ref{T:corollary2} is almost identical, and is omitted.

Finally we note the following fact about Carter-Payne homomorphisms between Specht modules. The proof is omitted, although it can easily be demonstrated by applying the methods used in the proof of Theorem \ref{T:Ln-generates}.

\begin{Proposition}\label{P:one-box-shift-restricted}
Suppose that $\lambda$ is a partition of $n+1$ and that $r$ is the $p$-residue of a removable node in $[\lambda]$.
Let $\lambda_1\triangleleft\ldots\triangleleft\lambda_m$ be all the partitions of $n$ whose Young diagram can be obtained
by removing a single node of $p$-residue $r$ from $[\lambda]$. Let $f_i$ be a non-zero $F\Sigma_n$-homomorphism $S^{\lambda_i}\rightarrow S^{\lambda_{i-1}}$, for each $i=2,\ldots,m$. Then for $1\leq j<i\leq m$, the composite homomorphism
$f_i\circ f_{i-1}\circ\ldots\circ f_{j+2}\circ f_{j+1}$ spans ${\rm Hom}_{F\Sigma_n}(S^{\lambda_i},S^{\lambda_{j}})$.
In particular, all these homomorphisms are non-zero.
\end{Proposition}

\section{Semi-standard homomorphisms}

We now discuss homomorphisms between Specht and permutation modules. If $\beta$ is a partition of\/ $n$, it is useful to describe the permutation module $M^\beta$ using certain $\alpha$-tableaux. Following James, an $\alpha$-tableau $T$ is of type $\beta$ if it has $\beta_i$ entries equal to $i$, for each $i\geq1$. The action of $\Sigma_n$ on $\alpha$-tableau restricts to an action on the $\alpha$-tableau of type $\beta$. Under this action, the $F$-span of the $\alpha$-tableau of type $\beta$ form a permutation module that is isomorphic to $M^\beta$.

Let $S$ and $T$ be $\alpha$-tableau of type $\beta$. We write
$$
S\approx T,\quad\mbox{if $S=T\pi$, for some $\pi\in R_t$.}
$$
Then $\approx$ is an equivalence relation of $\alpha$-tableau. We also say that $S$ is {\em row equivalent} to $T$. James defines the $\Sigma_n$ homomorphism $\theta_T:M^\alpha\rightarrow M^\beta$ by $\{t\}\,\theta_T:=\sum\limits_{S\approx T}S$.

A tableau $T$ is said to be {\em semi-standard} if the numbers are nondecreasing along the rows of\/ $T$ and strictly increasing down the columns of\/ $T$. Let $\widehat{\theta_T}$ denote the restriction of $\theta_T$ to $S^\alpha\subseteq M^\alpha$. James shows in \cite[13.13]{James} that the set
\begin{equation}\label{E:SStableau}
\{\widehat{\theta_T}\mid
   \mbox{$T$ a semi-standard $\alpha$-tableau of type $\beta$}\}
\end{equation}
forms a basis for Hom$_{F\Sigma_n}(S^\alpha,M^\beta)$. 

A useful result of James identifies $S^\beta$ as an intersection of the kernels of certain $\Sigma_n$-homomorphisms on $M^\beta$. Let $\beta$ be a partition of\/ $n$ that has $k$ parts. Then \cite[17.18]{James} shows that
\begin{equation}\label{E:SpechtIntersection}
S^\beta:=\bigcap\limits_{i=1}^{k-1}\,
         \bigcap\limits_{r=0}^{{\beta_{\,i+1}}-1}\,
		{\rm ker}\,\psi_{i,r}.
\end{equation}
Here for each pair $i,r$, let $\nu:=(\beta_1,\beta_2,\ldots,\beta_{i-1}, \beta_i+\beta_{i+1}-r,r,\beta_{i+2},\ldots)$. Then $\psi_{i,r}:M^\beta\rightarrow M^\nu$ is defined by setting $\{t\}\,\psi_{i,r}$ as the sum of all $\nu$-tabloids $\{s\}$ such that row $(i+1)$ of\/ $\{s\}$ is a subset of size $r$ of row $(i+1)$ of\/ $\{t\}$.

The action of\/ $\psi_{i,r}$ on $\alpha$-tableaux of type $\beta$ is particularly easy to describe. If $T$ is one such tableau then $T\,\psi_{i,r}$ is the sum of all $\alpha$-tableaux that can be obtained from $T$ by changing all but $r$ symbols $i+1$ to $i$. For example consider the effect of\/ $\psi_{1,2}$ on a $(4,3)$-tableau of type $(3^2,1)$:
$$
{\tiny\young(1112,223)}
\qquad
\raise1.2em\hbox{$\underrightarrow{\psi_{1,2}}$}\qquad
{\tiny\young(1111,223)}\quad
\raise1em\hbox{+}\quad
{\tiny\young(1112,123)}\quad
\raise1em\hbox{+}\quad
{\tiny\young(1112,213)}.
$$
In terms of semistandard homomorphisms, this gives
$$
\theta_{\tiny\young(1112,223)}\,\psi_{1,2}\,=\,4\,\theta_{\tiny\young(1111,223)}+\theta_{\tiny\young(1112,123)}.
$$

It is convenient to use the following concise notation
$$
\psi_i\,:=\,\psi_{i,\beta_{i+1}-1}.
$$
In particular $T\psi_i$ is the sum of all $\alpha$-tableaux of that can be obtained from $T$ by changing one symbol $i+1$ to an $i$.

For the rest of this paper let $a<b$ be integers such that $\alpha_a>\alpha_{a+1}$ and $\alpha_{b-1}>\alpha_b$. So $(a,\alpha_a)$ is a removable node of\/ $\alpha$ and $(b,\alpha_b+1)$ is an addable node of\/ $\alpha$. We fix the partition $\beta$ of\/ $n$ defined by
\begin{equation}\label{E:beta}
\beta_i=
\left\{
\begin{array}{ll}
\alpha_a-1,&\quad\mbox{if\/ $i=a$;}\\
\alpha_b+1,&\quad\mbox{if\/ $i=b$;}\\
\alpha_i,&\quad\mbox{if $i\ne a,b$.}\\
\end{array}
\right.
\end{equation}
We say that $\beta$ is a {\em one-box-shift} of\/ $\alpha$. Note that $\alpha\triangleright\beta$ in the dominance order.

Our first task is to enumerate all semi-standard $\alpha$-tableau of type $\beta$. So let $T$ be such a tableau. We claim that
\begin{equation}\label{E:one_box}
T(i,j)=i,\quad\mbox{for all $(i,j)\in[\alpha]$, unless $a\leq j<b$ and $j=\alpha_i$.}
\end{equation}
For, $T(i,j)\geq i$, as $T$ is column strict. So all symbols $i$ occur in the top $i$ rows of $T$. In particular, the symbols $1,\ldots,a-1$ occupy all entries in rows $1,\ldots,a-1$.

Now suppose that $a\leq i<b$. We prove by induction on $i$ that $T(i,j)=i$, for $1\leq j<\alpha_i$. The base case $i=a$ holds because the $\alpha_a-1$ symbols in $T$ must occupy all but the last entry in row $a$ of $T$ (there is no room for them further up). Now let $a<i$. Then there are $-1+\sum_{u=1}^{i-1}\alpha_u$ symbols in $T$ equal to one of $1,\ldots,i-1$. These occupy all but one of the $\sum_{u=1}^{i-1}\alpha_u$ nodes in the first $i-1$ rows of $T$. It follows that at most one symbol $i$ in $T$ does not belong to row $i$. This proves the inductive step.

Finally, suppose that $b\leq i$. Then there are $\sum_{u=1}^{i}\alpha_u$ symbols in $T$ equal to one of $1,\ldots,i$ and $\sum_{u=1}^{i-1}\alpha_u$ nodes in the first $i$ rows of $T$. It follows from this that every entry in the $i$-th row of $T$ is equal to $i$. This completes the proof of \eqref{E:one_box}.

Next we define
$$
\hat{T}(u):=T(u,\alpha_u),\quad\mbox{for $u=a,\ldots,b-1$.}
$$
Suppose that $a<i\leq b$. Then there are $\beta_i$-symbols in $T$ equal to $i$. At most one of them does not belong to the $i$-th row of $T$. Moreover, in view of \eqref{E:one_box}, if a symbol $i$ does not belong to row $i$, it occupies the end of a row between $a$ and $i-1$ inclusive. This shows that there exists a unique $u$ with $a\leq u\leq i$ such that $\hat{T}(u)=i$. It follows that $\hat{T}:\{a,\ldots,b-1\}\rightarrow\{a+1,\ldots,b\}$ is a bijection. Moreover, $\hat{T}(i)\geq i$, and $\hat{T}(i+1)>\hat{T}(i)$, if $\alpha_{i+1}=\alpha_i$. We call any such bijection $\hat{T}$ the {\em semi-standard $\alpha$-bijection of type $\beta$} associated with $T$.

We can recover the tableau $T$ from the bijection $\hat{T}$, using property \eqref{E:one_box}.

Let $u_1:=a$ and inductively define $u_i:=\hat{T}(u_{i-1})$, for $i=2,3\ldots$, if $u_{i-1}<b$. This gives a finite set of cardinality $m\geq1$:
$$
\{T\}:=\{u_1=a<u_2<\ldots<u_m\}.
$$
We claim that $\hat{T}(u)\ne u$ if and only if $u\in\{T\}$. We prove this by induction on $u\in\{a,\ldots,b-1\}$. For, suppose that $\hat{T}(u)\ne u$. Then there exists $v$ with $a\leq v<u$ and $\hat{T}(v)=u$. The inductive hypothesis implies that $v\in\{T\}$. So $v=u_i$, for some $i\geq 1$, whence $u=\hat{T}(u_i)=u_{i+1}$ also belongs to $\{T\}$. Note that $b=u_{m+1}$. If we temporarily extend $\hat{T}$ to a permutation of the set $\{a,\ldots,b\}$ via $\hat{T}(b):=a$, then this shows that $\hat{T}$ is the cycle permutation $(a=i_1,i_2,\ldots,i_m,i_{m+1}=b)$. So we can recover $\hat{T}$ from $\{T\}$.

In order for a subset $X$ of $\{a,\ldots,b-1\}$ to equal $\{T\}$ for some semi-standard $\alpha$-tableau of type $\beta$, it is necessary and sufficient that $a\in X$, and if $u\in X$ and $\alpha_u=\alpha_{u+1}$, then $u+1\in X$. We call any such $X$ a semi-standard $\alpha$-set of type $\beta$. We call $\{T\}$ the semi-standard $\alpha$-set of type $\beta$ associated with $T$ (or with $\hat{T}$).
We summarize the above discussion with:

\begin{Lemma}\label{L:bijection}
The associations $T\longleftrightarrow\hat{T}\longleftrightarrow\{T\}$ establish mutually inverse bijections between the semi-standard $\alpha$-tableau of type $\beta$, the semi-standard $\alpha$-bijections of type $\beta$, and the semi-standard $\alpha$-sets of type $\beta$.
\end{Lemma}

Given $T\leftrightarrow\hat{T}\leftrightarrow\{T\}$, we can use $\hat\theta_{\hat T}$ or $\hat\theta_{\{T\}}$ to denote $\hat\theta_T$. Moreover, we use $|T|$ or $|\hat{T}|$ to denote the cardinality of $\{T\}$.

If $\alpha$ has $m_i$ parts of length $i$, then there are $\prod_{\alpha_a<i<\alpha_b}(m_i+1)$ semistandard $\alpha$-tableau of type $\beta$. As an example let $\alpha=(4,3^2,2,1)$ and $\beta=(3^3,2,1^2)$. So $a=1,b=6$ and $m_3=2,m_2=1,m_1=1$. Then there are $12=(3+1)(1+1)(1+1)$ semistandard $\alpha$-tableau of type $\beta$. We list them along with the semistandard $\alpha$-sets of type $\beta$. We use $\cdot$ to indicate a symbol $i$ in row $i$:
\tiny
$$
\begin{array}{lllllll}
&\young(\cdot\cdot\cdot 6,\cdot\cdot 2,\cdot\cdot 3,\cdot 4,5)
&\young(\cdot\cdot\cdot 5,\cdot\cdot 2,\cdot\cdot 3,\cdot 4,6)
&\young(\cdot\cdot\cdot 4,\cdot\cdot 2,\cdot\cdot 3,\cdot 6,5)
&\young(\cdot\cdot\cdot 4,\cdot\cdot 2,\cdot\cdot 3,\cdot 5,6)
&\young(\cdot\cdot\cdot 3,\cdot\cdot 2,\cdot\cdot 6,\cdot 4,5)
&\young(\cdot\cdot\cdot 3,\cdot\cdot 2,\cdot\cdot 5,\cdot 4,6)\\
\vspace{1em}
&\{1\} &\{1,5\} &\{1,4\} &\{1,4,5\} &\{1,3\} &\{1,3,5\}\\
&\young(\cdot\cdot\cdot 3,\cdot\cdot 2,\cdot\cdot 4,\cdot 6,5)
&\young(\cdot\cdot\cdot 3,\cdot\cdot 2,\cdot\cdot 4,\cdot 5,6)
&\young(\cdot\cdot\cdot 2,\cdot\cdot 3,\cdot\cdot 6,\cdot 4,5)
&\young(\cdot\cdot\cdot 2,\cdot\cdot 3,\cdot\cdot 5,\cdot 4,6)
&\young(\cdot\cdot\cdot 2,\cdot\cdot 3,\cdot\cdot 4,\cdot 6,5)
&\young(\cdot\cdot\cdot 2,\cdot\cdot 3,\cdot\cdot 4,\cdot 5,6)\\
&\{1,3,4\} &\{1,3,4,5\} &\{1,2,3\} &\{1,2,3,5\} &\{1,2,3,4\} &\{1,2,3,4,5\}
\end  {array}
$$
\normalsize

For the rest of the paper we let
$$
\{R\,\}\,:=\,\{i\mid\mbox{ $(i,\alpha_i)$ is a removable node of $[\alpha]$, $a\leq i\leq b-1$}\}.
$$
The associated semi-standard $\alpha$-bijection of type $\beta$ is denoted by $\hat{R}$, and the tableau by $R$. A subset of $\{R\}$ that contains $a$ is called a {\em removable $\alpha$-set of type $\beta$}, or simply a {\em removable set}. Clearly each removable $\alpha$-set of type $\beta$ is also a semi-standard $\alpha$-set of type $\beta$. Each semi-standard set $\{T\}$ contains a largest removable set $\{N\}=\{T\}\cap\{R\}$.

We say that a pair of removable nodes are {\em adjacent} if there is no removable node of $[\alpha]$ between them. Likewise, if $i,j\in\{R\}$, we say that $i$ and $j$ are adjacent if the removable nodes $(i,\alpha_i)$ and $(j,\alpha_j)$ are adjacent; if $i<j$ this means that $\alpha_{i+1}=\alpha_j$.

\section{Relations between Semi-standard homomorphisms}

Throughout this section $T$ is a fixed semi-standard $\alpha$-tableau of type $\beta$.

The hooks-lengths in the $(\alpha_b+1)$-th column of $[\alpha]$ are $h_1,\ldots,h_{b-1}$, where
$$
h_i\,:=\,\alpha_i-\alpha_b+b-i-1,\quad\mbox{for $1\leq i\leq b-1$}.
$$
For convenience we set $h_i:=0$, for each $i\geq b$. Note that
\begin{equation}\label{E:hi-hj}
\begin{aligned}{}
h_i    -h_j\,&=\,(\alpha_i-\alpha_j)+(j-i),\quad&&\mbox{for $1\leq i,j\leq b-1$ and}\\
h_{b-1}-h_b\,&=\,\alpha_{b\!-\!1}-\alpha_b.\quad&&\\
\end  {aligned}
\end  {equation}

We begin with a useful technical result.

\begin{Lemma}\label{L:i+1row}
Let $U\approx T$, let $i\geq 0$ and let $V$ be an $\alpha$-tableau that is obtained from $U$ by changing a set of\/ $r\geq 1$ symbols $i+1$ to $i$, with at least one of the changes occurring in row $(i+1)$ of $U$. Suppose that $UC_t^-\ne 0$ and that $VC_t^-\ne 0$. Then $r=1$ and $\hat T(i)>i+1$ and $\hat{T}(i+1)=i+1$.
\begin{proof}
Let $c\geq 1$ be such that $U(i+1,c)=i+1$ and $V(i+1,c)=i$. As $UC_t^-\ne 0$ and $VC_t^-\ne 0$, separate applications of Lemma \ref{L:repeated} show that $U(i,c)\ne i+1$ and $V(i,c)\ne i$. Thus $U(i,c)\ne i,i+1$. As $U\approx T$, we deduce that $\hat{T}(i)>i+1$. It follows from this that $\hat{T}(i+1)=i+1$. In particular row $i+1$ contains all $\alpha_{i+1}$-symbols $i+1$ in $U$. Now $(i,c)$ is the unique node in row $i$ of $U$ that does not contain $i$. We conclude that $r=1$.
\end{proof}
\end{Lemma}

We will repeatedly, and without further comment, use the fact that $S^\alpha$ is generated by $e_t$ as a $\Sigma_n$-module. Moreover for all $i$ and $r$
\begin{equation}\label{E:e_t}
e_t\,\widehat{\theta_T}\,\psi_{i,r}\,=\,
  \sum\limits_{U\approx T}U\psi_{i,r}\,C_t^-,
\end{equation}
as both $\widehat{\theta_T}$ and $\psi_{i,r}$ are $\Sigma_n$-homomorphisms.

\begin{Lemma}\label{L:psi{ir}}
Let $i\geq 1$ and $1\leq r\leq\beta_{i+1}-1$ be such that $\widehat{\theta_T}\,\psi_{i,r}\ne0$. Then
$a\leq i<b$ and $r=\beta_{i+1}-1$.
\begin{proof}
Let $i\geq 0$, let $U\approx T$ and let $V$ be a tableau occurring in $U\psi_{i,r}$.

Suppose that $1\leq r<\beta_{i+1}-1$. As $T$ is semi-standard, Lemma \ref{L:bijection} implies that at
most one symbol $i+1$ in $U$ does not belong to row $(i+1)$. But $V$ is the result of changing $\beta_{i+1}-r\geq 2$ symbols in $U$ from $i+1$ to $i$. So one of the changes occurs in row $(i+1)$ of $U$. It then follows from Lemma \ref{L:i+1row} that either $U\,C_t^-=0$ or $V\,C_t^-=0$.

Suppose that $i<a$ or $b<i$. Then all symbols $i+1$ in $U$ belong to row $(i+1)$ of $U$, and Lemma \ref{L:i+1row} implies that $U\,C_t^-=0$ or $V\,C_t^-=0$.

In all cases we have $U\,C_t^-\,\psi_i=U\,\psi_i\,C_t^-=\sum V\,C_t^-=0$. We then use \eqref{E:e_t} to conclude that $\widehat{\theta_T}=0$, which is contrary to hypothesis.
\end{proof}
\end{Lemma}

Lemmas \ref{L:inot}, \ref{L:i-1bad} and \ref{L:i-1good} will show that the image of $\widehat{\theta_T}$ under $\psi_i$ is either zero or a specified non-zero integer multiple of a semi-standard homomorphism from $S^\alpha$ into the permutation module $M^\gamma$. Here $\gamma$ is the composition
$$
\gamma_j=\left\{
\begin{array}{ll}
\beta_i+1,&\quad\mbox{if $j=i$;}\\
\beta_{i+1}-1,&\quad\mbox{if $j=i+1$;}\\
\beta_j,&\quad\mbox{otherwise.}
\end{array}\right.
$$
Notice that $\beta\triangleleft\gamma\trianglelefteq\alpha$.

\begin{Lemma}\label{L:inot}
Let $a<i\leq b-1$ be such that $i\not\in\{T\}$. Then
$$
\widehat{\theta_T}\,\psi_i\,=\,
\left\{
\begin{aligned}{}
&\widehat{\theta_W},	&&\quad\mbox{if\/ $i+1\in\{T\}$;}\\
&0,			&&\quad\mbox{otherwise.}
\end{aligned}
\right.
$$
Here $W$ is obtained from $T$ by changing the single $i+1$ in row $\hat{T}^{-1}(i+1)$ to an $i$.
\begin{proof}
Let $U\approx T$. The hypothesis on $i$ implies that the $\alpha_i$ symbols $i$ occupy the $\alpha_i$ columns of row $i$ in $U$. So each symbol $i+1$ which belongs to row $i+1$ of $U$ shares its column with a symbol $i$ in row $i$ of $U$. Let $a\leq j<b$ be such that $\hat{T}(j)=i+1$.

Suppose first that $j=i+1$. Then all symbols $i+1$ in $U$ occur in row $(i+1)$. It then follows from Lemma \ref{L:i+1row} and the first paragraph that $U\,\psi_i\,C_t^-\,=\,0$. We conclude from \eqref{E:e_t} that $\widehat\theta_T\,=\,0$ in this case.

From now on assume that $j\ne i+1$, whence $j<i$. Then exactly one of the symbols $i+1$ in $U$ occurs in row $j$, and all the other symbols $i+1$ belong to row $(i+1)$ of $U$. It then follows from Lemma \ref{L:i+1row} and the first paragraph that either $U\,C_t^-\,=\,0$ or $U\,\psi_i\,C_t^-\,=\,V\,C_t^-$, where $V$ is the tableau obtained from $U$ by changing the symbol $i+1$ in row $j$ to an $i$. It is obvious that $V\approx W$ and also that $U$ is the only tableau that is row equivalent to $T$ such that $U\,\psi_i\,C_t^-\,=\,V\,C_t^-$.

Conversely begin with $V\approx W$. Then $V$ has a unique $i$ in row $j$. Changing this $i$ to an $i+1$ produces a tableau $U$ that is row equivalent to $T$. Moreover clearly $U\,\psi_i\,C_t^-\,=\,V\,C_t^-$.

These arguments and \eqref{E:e_t} show that $e_t\,\widehat{\theta_T}\,\psi_i\,=\,\sum_{V\approx W}V\,C_t^-\,=\,
 e_t\,\widehat{\theta_W}$. We conclude that in this case $\widehat{\theta_T}\,\psi_i\,=\,\widehat{\theta_W}$.
This completes the proof.
\end{proof}
\end{Lemma}

\begin{Lemma}\label{L:i-1bad}
Let $a<i\leq b-1$ be such that $\alpha_{\hat{T}^{-1}(i)}=\alpha_i$. Then
$$
\widehat{\theta_T}\,\psi_i\,=0.
$$
\begin{proof}
The hypothesis implies that $\hat{T}(i-1)=i$ and $\alpha_{i-1}=\alpha_i$. Let $U\approx T$ be such that $U\,C_t^-\,\ne\,0$. Then one symbol $i$ belongs to row $(i-1)$ of $U$, and the remaining $\alpha_i-1$ symbols $i$ belong to row $i$ of $U$. It then follows from Lemma \ref{L:repeated} that each of the first $\alpha_i$ columns of $U$ contains exactly one symbol $i$.

Now either $\hat{T}(i)=i+1$ or $\hat{T}(i+1)=i+1$. In either case all symbols $i+1$ belongs to the first $\alpha_i$ columns of $U$. It then follows from the previous paragraph that each symbol $i+1$ in $U$ belongs to a column that also contains a symbol $i$. We deduce from Lemma \ref{L:repeated} that $U\,\psi_i\,C_t^-\,=\,0$. This contradiction means that our supposition that $U$ exists is false.

Using \eqref{E:e_t}, we see that $e_t\widehat{\theta_T}\psi_i=0$. We conclude that $\widehat{\theta_T}\,\psi_i\,=0$.
\end  {proof}
\end{Lemma}

Suppose that $a\leq i\leq b-1$ is such that $\hat{T}(i)\ne i$. Let $k\geq 1$ be such that $(i+k,\alpha_{i+k})$ is a removable node of $\alpha$, and $\alpha_{i+1}=\alpha_{i+k}$. If $\hat{T}(i)=i+1$, set $T\vee i:=T$. Otherwise, set $\{T\vee i\}:=\{T\}\cup\{i+1,\ldots,i+k\}$. Then $\{T\vee i\}$ is the smallest semi-standard $\alpha$-set of type $\beta$ that contains $i+1$ and $\{T\}$. There is a corresponding $\alpha$-tableau $T\vee i$ and $\alpha$-bijection $\widehat{T\vee i}$:
\begin{equation}
\widehat{T\vee i}(r)=
\left\{
\begin{array}{ll}
r+1,&\quad\mbox{if $i\leq r<i+k<b$;}\\
\hat{T}(i),&\quad\mbox{if $r=i+k$ and $\hat{T}(i+k)=i+k$;}\\
\hat{T}(r),&\quad\mbox{otherwise.}
\end{array}\right.
\end  {equation}
Suppose that $T$ is such that $\hat{T}(i+k)=i+k$. In particular $(i,\alpha_i)$ is a removable node of $\alpha$. The following describe the semi-standard $\alpha$-sets $\{S\}$ such that $S\vee i=T\vee i$:
$$
\begin{array}{ccrr}
\{T\}&&\\
\{T\}&\cup&\{i+k\}\\
\{T\}&\cup&\{i+k-1,i+k\}\\
\vdots&\vdots&\vdots\\
\{T\}&\cup&\{i+2,\ldots,i+k-1,i+k\}\\
\{T\}&\cup&\{i+1,i+2,\ldots,i+k-1,i+k\}&=\{T\vee i\}\\
\end{array}
$$

\begin{Lemma}\label{L:i-1good}
Let $a\leq i\leq b-1$ be such that $i=a$ or $\alpha_{\hat{T}^{-1}(i)}\ne\alpha_i$. Then
$$
\widehat{\theta_T}\,\psi_i\,=\,
\left\{
\begin{aligned}{}
&&(h_i-h_{i+1})
	&\hspace{.1in}\widehat{\theta_W},
		&&\hspace{.1in}\mbox{if\/ $\hat{T}(i)=i+1$;\,}\\
&&\,\,(-1)^{|\{T\vee i\}\backslash\{T\}|}
	&\hspace{.1in}\widehat{\theta_W},
		&&\hspace{.1in}\mbox{otherwise}.
\end{aligned}
\right.
$$
Here $W$ is obtained from $T\vee i$ by changing the unique symbol $i+1$ in row $i$ to an $i$.
\begin{proof}
Let $U\approx T$ be such that $U\,\psi_i\,C_t^-\ne 0$.

We consider first the case that $\hat{T}(i)=i+1$, or equivalently $T=T\vee i$. In particular there exists a column number $c$ so that node $(i,c)$ of $U$ contains the symbol $i+1$, and all other symbols $i+1$ belong to row $(i+1)$ of $U$. Lemma \ref{L:i+1row} then implies that $U\,\psi_i\,C_t^-=V\,C_t^-$, where $V$ is obtained from $U$ by changing the $i+1$ in node $(i,c)$ to an $i$. In particular $V\approx W$.

Lemma \ref{L:repeated} implies that column $c$ of $U$ does not contain two symbols $i+1$. So either $(i,c)$ is to the right of all nodes in row $i+1$ i.e. $c\in\{\alpha_{i+1}+1,\ldots,\alpha_i\}$, or $i+1\ne b$ and $(i+1,c)$ is the unique node in row $i+1$ of $U$ that does not contain $i+1$. This argument and \eqref{E:hi-hj} shows that there are $\alpha_{i}-\alpha_{i+1}+1=h_i-h_{i+1}$ choices for a tableau $U$ such that $U\psi_i\,C_t^-=V\,C_t^-$. It then follows from  \eqref{E:e_t} that
$$
e_t\,\widehat{\theta_T}\,\psi_i\,=\,
(h_i-h_{i+1})\sum\limits_{V\approx W}V\,C_t^-\,=\,
(h_i-h_{i+1})e_t\,\widehat{\theta_W}.
$$
We conclude that in this case $\widehat{\theta_T}\,\psi_i\,=\, (h_i-h_{i+1})\widehat{\theta_W}$.

Now consider the case that $\hat{T}(i)\ne i+1$. Then $\hat{T}(i+1)=i+1$. As $\hat{T}(i)\ne i$, this ensures that $(i,\alpha_i)$ is a removable node of $\alpha$. Set $j:=\hat{T}(i)$. Lemma \ref{L:bijection} shows that there exists $c\geq 1$ such that node $(i,c)$ of $U$ contains the symbol $i+j>i+1$, and all other nodes in row $i$ contain the symbol $i$. Furthermore, all symbols $i+1$ in $U$ occur in row $(i+1)$ of $U$. Then Lemma \ref{L:repeated} implies that $c\leq\alpha_{i+1}$ and $U\psi_iC_t^-=VC_t^-$, where $V$ is obtained from by changing the $i+1$ in node $(i+1,c)$ of $U$ to an $i$.

Let $k\geq 1$ be such that $\alpha_{i+k}=\alpha_{i+1}$ and $(i+k,\alpha_{i+k})$ is a removable node of $\alpha$. Let $\pi$ be the permutation in $\Sigma_n$ that cycles the entries in the nodes of $U$ as follows:
$$
\begin{aligned}{}
\mbox{if $j>k$ then }\,
&\pi:\hspace{1.5em}(i+k,c)\rightarrow\ldots\rightarrow(i+1,c)\rightarrow(i,c)\rightarrow(i+k,c);\\
\mbox{if $j\leq k$ then }\, &\pi:(i+j-1,c)\rightarrow\ldots\rightarrow(i+1,c)\rightarrow(i,c)\rightarrow(i+j-1,c).
\end{aligned}
$$ 
Then $\pi$ belongs to $C_t$. So $\pi\,C_t^-=\sgn(\pi)\,C_t^-$. Also $V\pi\approx W$ is strictly increasing down columns, and we obtain each $\alpha$-tableau that is row equivalent to $W$ as an $V\pi$ exactly once, by an appropriate choice of $U\approx T$.

Note that
$$
\{T\vee i\}=\left\{
\begin{array}{ll}
\{T\}\dot{\cup}\{i+1,\ldots,i+k\},&\quad\mbox{if $j>k$,}\\
\{T\}\dot{\cup}\{i+1,\ldots,i+j-1\},&\quad\mbox{if $j\leq k$.}
\end{array}
\right.
$$
So $\sgn(\pi)=(-1)^{|\{T\vee i\}\backslash\{T\}|}$ does not depend on the choice of $U\approx T$. Thus
$$
e_t\,\widehat{\theta_T}\,\psi_i\,=\,
\sum\limits_{U\approx T}U\,\psi_i\,C_t^-\,=\,
(-1)^{|\{T\vee i\}\backslash\{T\}|}\sum\limits_{V\approx W}V\,C_t^-\,=\,
(-1)^{|\{T\vee i\}\backslash\{T\}|}e_t\,\widehat{\theta_W}.
$$
It follows that in this case
$\widehat{\theta_T}\,\psi_i\,=\,
	(-1)^{|\{T\vee i\}\backslash\{T\}|}\widehat{\theta_W}$.
\end{proof}
\end{Lemma}

\section{Carter-Payne homomorphisms}

We now fix a $\Sigma_n$-homomorphism $\hat\theta:S^\alpha\rightarrow S^\beta$. By \eqref{E:SStableau} we may write
\begin{equation}\label{E:theta}
\hat\theta\,=\,
\sum_T\,\Lambda_T\,\hat\theta_T,
\qquad\mbox{for certain $\Lambda_T\in F$,}
\end  {equation}
where $T$ ranges over the semi-standard $\alpha$-tableaux of type $\beta$. We aim to show that, up to a scalar multiple, there is at most one such homomorphism $\hat\theta$. Since  the image of $\hat{\theta}$ is a submodule of $S^\beta$, it follows that  $\hat{\theta} \psi_i = 0$ for all $i$. Fix an $i$. Lemmas \ref{L:inot}, \ref{L:i-1bad} and \ref{L:i-1good} show that for all $T$, $\hat{\theta}_T \psi_i$ is equal to $a_{T,i}\hat{\theta}_{S_{T,i}}$, where $a_{T,i}$ is an integer and $S_{T,i}$ is a semistandard tableau depending on $T$ and $i$. Maps associated to semistandard tableaus are  linearly independent. Thus the equation  $\hat{\theta}\psi_i = 0$  produces several equations that the coefficients $\Lambda_T$ must satisfy; for any semistandard tableau $W$, we have $\sum\Lambda_T\,a_{T,i}=0$, where the sum is over all $T$ such that $S_{T,i}=W$. We shall collect enough of these equations to show that the space of all $\Sigma_n$-homomorphisms from $S^\alpha$ to $S^\beta$ is at most one-dimensional.

\begin{Lemma}\label{L:removable}
Let $\{T\}$ be a semi-standard $\alpha$-set of type $\beta$. Then
$$
(\!-\!1)^{|\{T\}|}\Lambda_{\{T\}}=(\!-\!1)^{|\{T\}\cap\{R\}|}\Lambda_{\{T\}\cap\{R\}}.
$$ 
\begin{proof}
The result is trivial if $\{T\}\subseteq\{R\}$. We prove the result by induction on $|\{T\}\backslash\{R\}|$. So we may assume that $\{T\}\backslash\{R\}$ is nonempty. In particular, we may choose $i\in\{a+1,\ldots,b-1\}$ such that $\alpha_{\hat{T}^{-1}(i)}>\alpha_i$ and $\alpha_i=\alpha_{i+1}$. This forces $\hat{T}(i)=i+1$. In terms of tableaux, this can be visualized as follows. The set $\{T\}$ is associated to a series of vertical strips of entries  at the right edge of the tableau, vertical strips whose lowest entry is at a removable node. We are focussing on an entry at the top of one of these strips. The inductive step consists of proving that $\Lambda_T+\Lambda_{\{T\}\backslash\{i\}}=0$. Geometrically the end result of this lemma is to show that we need only understand the coefficients of  semistandard sets associated to unions of strips of length 1 containing a single removable node.

Lemma \ref{L:i-1good} implies that $\widehat{\theta_T}\psi_i=\widehat{\theta_W}$, where $W$ is the $\alpha$-tableau that is obtained from $T$ by changing the single $i+1$ in row $i$ of $T$ to an $i$.

Suppose that $U$ is a semi-standard $\alpha$-tableau of type $\beta$ such that $\widehat{\theta_U}\psi_i$
is a non-zero multiple of $\widehat{\theta_W}$. Lemmas \ref{L:inot}, \ref{L:i-1bad} and \ref{L:i-1good}
imply that either $i\not\in\{U\}$ or (as $\alpha_i=\alpha_{i+1}$) both $i$ and $i+1$ belong to $\{U\}$. Moreover $\{T\}$ and $\{U\}$ do not differ apart from in the set $\{i,i+1\}$. In the former case $\{U\}=\{T\}\backslash\{i\}$ and hence $\widehat{\theta_U}\psi_i=\widehat{\theta_W}$. In the latter case $\{U\}=\{T\}$. This completes the proof of the inductive step, and hence the proof of the lemma also.
\end{proof}
\end{Lemma}

From now on we concentrate on finding relations among the $\Lambda_N$, with $\{N\}\subseteq\{R\}$ a removable $\alpha$-set of type $\beta$. So we consider $N\leftrightarrow\hat{N}\leftrightarrow\{N\}$ such that $\{N\}$ is an arbitrary removable $\alpha$-set of type $\beta$.

\begin{Lemma}\label{L:Tij}
Let $i\in\{N\}$ and $j=\hat{N}(i)$ be such that $(i,\alpha_i)$ and $(j,\alpha_j)$ are adjacent removable nodes of $[\alpha]$. Then
$$
h_i\,\Lambda_{\{N\}}\,+\,\Lambda_{\{N\}\backslash\{i\}}\,=\,
h_j\,\Lambda_{\{N\}}\,+\,\Lambda_{\{N\}\backslash\{j\}}.
$$
The term involving $\{N\}\backslash\{i\}$ is omitted if\/ $i=a$, while the term involving $\{N\}\backslash\{j\}$ is omitted if\/ $i=b-1$ (and so $j=b$; recall also that $h_b=0$).
\begin{proof}
We assume the generic case that $i\ne a$ and $j\ne b$. The two exceptional cases follow from similar arguments.

Lemma \ref{L:i-1good} implies that $\widehat{\theta_N}\psi_i$ is a non-zero multiple of $\widehat{\theta_W}$. Here $W$ is the tableau that is obtained by changing the single $i+1$ in row $i$ of\/ $N\vee i$ to an $i$. Note that $\{N\vee i\}=\{N\}\dot{\cup}\{i+1,\ldots,j-1\}$. We enumerate all standard $\alpha$-tableau $T$ of type $\beta$ such that $\widehat\theta_T\psi_i$ is a non-zero multiple of $\widehat\theta_W$. 

Suppose that $\hat{T}(i)=i$. Then $i\ne a$. Lemma \ref{L:inot} forces $\hat{T}(i+1)\ne i+1$ and $\widehat{\theta_T}\psi_i=\widehat{\theta_W}$. Also $\{T\}=\{T\vee i\}=\{N\vee i\}\backslash\{i\}$. Thus $\{T\}\cap\{R\}=\{N\}\backslash\{i\}$ and $|\{T\}\backslash\{T\}\cap\{R\}|=j-i-1$. Lemma \ref{L:removable} then implies that in this case
\begin{equation}\label{E:one}
            \Lambda_{\{T\}}\widehat{\theta_T}\psi_i\,=\,
(-1)^{j-i-1}\Lambda_{\{N\}\backslash\{i\}}\widehat\theta_W.
\end{equation}

From now on we assume that $\hat{T}(i)\ne i$. Then Lemma \ref{L:i-1good} forces $\{T\vee i\}=\{N\vee i\}$. There are three cases depending on whether $i+1,j\in\{T\}$ or not.

Suppose first that $i+1\in\{T\}$. Then $j\in\{T\}$ and $T=T\vee i$, and Lemma \ref{L:i-1good} shows that $\widehat{\theta_T}\psi_i=(h_i-h_{i+1})\widehat{\theta_W}$. But $\{T\}\cap\{R\}=\{N\}$. So $|\{T\}\backslash\{T\}\cap\{R\}|=j-i-1$. We conclude from Lemma \ref{L:removable} that in this case
\begin{equation}\label{E:three}
\Lambda_{\{T\}}\widehat{\theta_T}\psi_i\,=\,
(-1)^{j-i-1}(h_i-h_{i+1})\Lambda_{\{N\}}\widehat{\theta_W}.
\end{equation}

There are $j-i-1$ semi-standard sets $\{T\}$ such that $j\in\{T\}$ but $i+1\not\in\{T\}$, one for each of the rows $i+1,i+3,\ldots,j$. For each such $\{T\}$ we have $\{T\}\cap\{R\}=\{N\}$. In particular $|\{T\vee i\}\backslash\{T\}\cap\{R\}|=j-i-1$. It then follows from Lemmas \ref{L:i-1good} and \ref{L:removable} that
\begin{equation}\label{E:five}
\Lambda_{\{T\}}\widehat\theta_T\psi_i=
 (-1)^{j-i+1}\Lambda_{\{N\}}\widehat{\theta_W}.
\end{equation}

The final possibility is that $i+1,j\not\in\{T\}$. Then $\{T\}=\{N\}\backslash\{j\}$ is a removable set and $|\{T\vee i\}\backslash\{T\}\cap\{R\}|=j-i$. So by Lemmas \ref{L:i-1good} and \ref{L:removable} we have
\begin{equation}\label{E:six}
\Lambda_{\{T\}}\widehat{\theta_T}\psi_i\,=\,
(-1)^{j-i}\Lambda_{\{N\}\backslash\{j\}}\widehat{\theta_W}. 
\end{equation}

We now apply $\psi_i$ to $\hat\theta$, expand as a linear combination of semi-standard homomorphisms, and examine the coefficient of $\widehat{\theta_W}$ in this expansion. Multiplying \eqref{E:one}, \eqref{E:three}, \eqref{E:five} and \eqref{E:six} by $(-1)^{j-i-1}$, we get
$$
                      \Lambda_{\{N\}\backslash\{i\}}
+(h_i-h_{i+1})\Lambda_{\{N\}}
-                     \Lambda_{\{N\}\backslash\{j\}}
+(j-i-1)              \Lambda_{\{N\}}\,=\,0.
$$
Using \eqref{E:hi-hj} and the fact that $\alpha_{i+1}=\alpha_j$, we can rewrite this as
$$
(h_i-h_j)\Lambda_{\{N\}}\,+\,
         \Lambda_{\{N\}\backslash\{i\}}\,=\,
         \Lambda_{\{N\}\backslash\{j\}}.
$$
\end{proof}
\end{Lemma}

We now prove the main result of this section.

\begin{Theorem}\label{T:Lambda}
$(-1)^{|N|}\,\Lambda_N\,=\,
 (-1)^{|R|}\big(\prod_{u\in\{R\}\backslash\{N\}}\,h_u\big)
		\,\Lambda_R.$
\begin{proof}
We shall prove this theorem by repeated applications of Lemma \ref{L:Tij}, with $i\ne a$. To do this, we use a partial order on removable $\alpha$-tableau of type $\beta$. Suppose that $\{S\},\{T\}\subseteq\{R\}$, with $a\in\{S\},\{T\}$. We write $\{S\}\lhd\{T\}$ if $\{S\}=\{T\}\backslash\{i\}$, for some $i$, or if there exists $(i,\alpha_i)$ and $(j,\alpha_j)$ adjacent removable nodes of $[\alpha]$ such that $i>j$ and $\{S\}\backslash\{T\}=\{i\}$ and $\{T\}\backslash\{S\}=\{j\}$. Let $<$ be the partial order generated by $\lhd$. We shall prove the theorem by induction on $<$. The base case is $\{N\}=\{R\}$, when the conclusion is trivially true. Assume then that $\{N\}\subsetneq\{R\}$.

Suppose first that $b-1\not\in\{N\}$. We apply Lemma \ref{L:Tij} to $\{N\}\cup\{b-1\}$ to get
$$
\Lambda_{N}=(-h_{b-1})\,\Lambda_{\{N\}\cup\{b-1\}}.
$$
But $\{N\}\lhd\{N\}\cup\{b-1\}$. So the inductive hypothesis gives
$$
\Lambda_{\{N\}\cup\{b-1\}}=\prod_{u\in\{R\}\backslash\{N\}\cup\{b-1\}}\hspace{-.1in}(-h_u)\,\Lambda_{R}.
$$
The claimed expression for $\Lambda_N$ follows in this case.

Suppose then that $b-1\in\{N\}$. Then we can find adjacent removable nodes $(i,\alpha_i)$ and $(j,\alpha_j)$ of $[\alpha]$ such that $i>j>a$ and $i\in\{N\}$ but $j\not\in\{N\}$. Lemma \ref{L:Tij} implies that
$$
         \Lambda_{\{N\}}\,=\,
(h_j-h_i)\Lambda_{\{N\}\cup\{j\}}\,+\,
         \Lambda_{\{N\}\cup\{j\}\backslash\{i\}}.
$$
But $\{N\}\lhd\{N\}\cup\{j\}$ and $\{N\}\lhd\{N\}\cup\{j\}\backslash\{i\}$. So by the induction hypothesis we have
$$
\begin{aligned}
 ( h_j-h_i)\Lambda_{\{N\}\cup\{j\}}\,&=\,
&(-h_i+h_j)\hspace{.0in}&\prod_{u\in\{R\}\backslash\{N\}\cup\{j\}}
	\hspace{-.1in}(-h_u)\,\Lambda_{R};\\
\Lambda_{\{N\}\cup\{j\}\backslash\{i\}}\,&=\,
&(-h_j)\hspace{.0in}&\prod_{u\in\{R\}\backslash\{N\}\cup\{j\}}
	\hspace{-.1in}(-h_u)\,\Lambda_{R}.
\end  {aligned}
$$
Adding, we obtain the inductive step.
\end  {proof}
\end{Theorem}

Our corollary is a precise form of Theorem \ref{T:main}.

\begin{Corollary}\label{C:explicit}
The $F$-space ${\rm Hom}_{F\Sigma_n}(S^\alpha,S^\beta)$ is at most one dimensional. Every non-zero $\Sigma_n$-homomorphism $S^\alpha\rightarrow S^\beta$ is a scalar multiple of
$$
\sum_T\bigg((-1)^{|T|}\hspace{-.6em}\prod_{u\in\{R\}\backslash\{T\}}\hspace{-.3em} h_u\bigg)\,\widehat{\theta_T}.
$$
\begin{proof}
This follows at once from Lemma \ref{L:removable} and Theorem \ref{T:Lambda}.
\end  {proof}
\end  {Corollary}

From now on we assume that
$$
\hat\theta=\sum_T\bigg((-1)^{|T|}\hspace{-.6em}\prod_{u\in\{R\}\backslash\{T\}}\hspace{-.3em} h_u\bigg)\,\widehat{\theta_T}.
$$

\begin{Lemma}\label{L:all_but_a}
Over ${\mathbb Z}$ we have $\hat\theta\,\psi_{i,r}=0$, unless $i=a$ and $r=\beta_{a+1}-1$.
\begin{proof}
In view of Lemma \ref{L:psi{ir}}, we need only show that $\hat\theta\,\psi_i=0$, for $i=a+1,\ldots,b-1$. Fix $i$ with $a<i<b$, and let $W$ be an $\alpha$-tableau such that there exists a semi-standard $\alpha$-tableau $X$ of type $\beta$ such that $\widehat{\theta_X}\psi_i$ is a nonzero multiple of $\widehat{\theta_W}$. We enumerate all such tableau $X$.

Suppose first that $(i,\alpha_i)$ is not a removable node of $[\alpha]$. Then Lemmas \ref{L:psi{ir}}, \ref{L:inot}, \ref{L:i-1bad} and \ref{L:i-1good} show that there is a semi-standard $\alpha$-tableau $T$ of type $\beta$ with $i\not\in\{T\}$ such that $X=T$ or $S$, where $\{S\}:=\{T\}\cup\{i\}$. Moreover, Lemmas \ref{L:inot} and \ref{L:i-1good} and the definition of $\theta$ show that 
$$
(\Lambda_T\widehat{\theta_T}+\Lambda_S\widehat{\theta_S})\psi_i=(-1)^{|T|}(1-1)\big(\hspace{-.3em}\prod_{u\in\{R\}\backslash\{T\}}h_u\big)\widehat{\theta_W}=0.
$$

Suppose then that $(i,\alpha_i)$ is a removable node of $[\alpha]$. Let $j>i$ be such that $(j,\alpha_j)$ is a removable node of $[\alpha]$ and $\alpha_j=\alpha_{i+1}$. So $(i,\alpha_i)$ and $(j,\alpha_j)$ are adjacent removable nodes. Then Lemmas \ref{L:psi{ir}}, \ref{L:inot}, \ref{L:i-1bad} and \ref{L:i-1good} show that there is a semi-standard $\alpha$-tableau $T$ of type $\beta$ with $i,i+1\in\{T\}$, such that $X=T\vee i$ or $\{X\}=\{T\}\backslash\{i\}$. Thus $\{X\}$ is one of the following semi-standard $\alpha$-sets:
$$
\begin{array}{rcl}
\{T\},     &&\\
\{S_i\}    &:=&\{T\}\backslash\{i\},\\
\{S_{i+1}\}&:=&\{T\}\backslash\{i+1\},\\
\{S_{i+2}\}&:=&\{T\}\backslash\{i+1,i+2\},\\
\vdots&&\vdots\\
\{S_{j-1}\}&:=&\{T\}\backslash\{i+1,i+2,\ldots,j-1\},\\
\{S_j\}    &:=&\{T\}\backslash\{i+1,i+2,\ldots,j-1,j\}.
\end{array}
$$
Lemmas \ref{L:inot} and \ref{L:i-1good} and show that the coefficient of $\widehat{\theta_W}$ in $\theta\,\psi_a$ is given by
$$
(\Lambda_T\widehat{\theta_T}+\sum_{u=i}^{j}\Lambda_{S_u}\widehat{\theta_{S_u}})\psi_i=(-1)^{|T|}((h_i-h_{i+1})-h_i+(j-i-1)+h_j)\big(\hspace{-.8em}\prod_{u\in\{R\}\backslash\{T\}}\hspace{-.6em}h_u\big)\widehat{\theta_W}=0,
$$
as $h_j-h_{i+1}=i+1-j$, by \eqref{E:hi-hj} using the fact that $\alpha_j=\alpha_{i+1}$.
\end{proof}
\end{Lemma}

Let $\gamma$ be the partition corresponding to the following composition of $n$:
$$
\alpha_1,\ldots,\alpha_a,\alpha_{a+1}-1,\alpha_{a+2},\ldots,\alpha_{b-1},\alpha_b+1,\alpha_{b+1},\ldots.
$$
We now deal with the exceptional test function $\psi_a$.

\begin{Lemma}\label{L:the_a_case}
Over ${\mathbb Z}$ we have $\operatorname{Im}\hat\theta\,\psi_a\subseteq h_aM^\gamma$.
\begin{proof}
We retain the notation of $W,T$ from the previous lemma, with $i=a$ and $j$ such that $(j,\alpha_j)$ be the highest removable node of $[\alpha]$ below $(a,\alpha_a)$. Equivalently $\alpha_{j+1}<\alpha_j$ and $\alpha_j=\alpha_{a+1}$. The semi-standard sets that contribute to $\widehat{\theta_W}$ can be enumerated as $\{T\},\{S_{i+1}\},\ldots,\{S_j\}$ i.e. the same as in Lemma \ref{L:all_but_a}, but with $\{S_i\}=\{T\}\backslash\{a\}$ omitted (because it does not contain $a$, and hence is not a semi-standard $\alpha$-set of type $\beta$). So the coefficient of $\widehat{\theta_W}$ in $\theta\,\psi_a$ is given by
$$
\begin{array}{ll}
(\Lambda_T\widehat{\theta_T}+{\sum\limits_{u=a+1}^{j}\hspace{-.5em}\Lambda_{S_u}\widehat{\theta_{S_u}}})\psi_a
&\hspace{-1em}=(-1)^{|T|}((h_a\!-\!h_{a+1})+({\!j\!-\!a\!-\!1})+h_j)\big(\hspace{-.8em}\prod\limits_{u\in\{R\}\backslash\{T\}}\hspace{-.6em}h_u\big)\widehat{\theta_W}\\
&=(-1)^{|T|}h_a\big(\prod_{u\in\{R\}\backslash\{T\}}h_u\big)\widehat{\theta_W}.
\end{array}
$$
\end{proof}
\end{Lemma}

We also obtain the following criterion for the existence of a non-zero homomorphism. It is a special case of a result of Carter and Payne \cite{CarterPayne}. The `if' part is also true over fields of characteristic $2$.

\begin{Corollary}\label{C:main}
Let $\beta$ be a one-box-shift of $\alpha$, with $\beta_a=\alpha_a-1$ and $\beta_b=\alpha_b+1$, for $a<b$. Then ${\rm Hom}_{F\Sigma_n}(S^\alpha,S^\beta)\ne 0$ if and only if\/ $\alpha_a-a\equiv\beta_b-b$ $\mod p$.
\begin{proof}
Lemmas \ref{L:all_but_a} and \ref{L:the_a_case} show that, over the field $F$ of characteristic $p$, the image of $\theta$ belongs to $S^\beta$ if and only if $h_a\equiv0$ mod $p$ (for the only if, for $T=R$ we have $\theta\psi_a\ne0$ if $p\not|h_a$ as in that case$\prod_{u\in\{R\}\backslash\{T\}}h_u=1$). The result follows, as $h_a=(\alpha_a-\alpha_b)+(b-a-1)=(\alpha_a-a)-(\beta_b-b)$. 
\end  {proof}
\end{Corollary}

\section{Jantzen Filtration}\label{S:Jantzen}

We let $\langle\,,\,\rangle$ denote the symmetric bilinear form for which the $\beta$-tabloids form an orthonormal basis of $M^\beta$. This form is $\Sigma_n$-invariant. Over ${\mathbb Q}$ we have a direct sum decomposition:
$$
M_{\mathbb Q}^\beta=S_{\mathbb Q}^\beta\oplus S_{\mathbb Q}^{\beta\perp},
$$
where $S_{\mathbb Q}^{\beta\perp}=\{m\in M_{\mathbb Q}^\beta\mid\langle m,S_{\mathbb Q}^\beta\rangle=0\}$ is the perpendicular space to $S^\beta$. Now $S_{\mathbb Z}^\beta\oplus S_{\mathbb Z}^{\beta\perp}$ is a full sublattice of $M_{\mathbb Z}$ (in the sense that both have the same ${\mathbb Z}$-rank). However the quotient lattice $M_{\mathbb Z}^\beta/S_{\mathbb Z}^\beta\oplus S_{\mathbb Z}^{\beta\perp}$ can be quite large. Over a field $F$ of characteristic $p$, the restriction of the form $\langle\,,\,\rangle$ to $S^\beta$ is generally degenerate. Over ${\mathbb Z}$, the module $S_{\mathbb Z}^\beta$ has a sequence of Jantzen submodules:
$$
J^i(S_{\mathbb Z}^\beta):=\{x\in S_{\mathbb Z}^\beta\mid\langle x,S_{\mathbb Z}^\beta\rangle\subseteq p^i{\mathbb Z}\}.
$$
The images of these modules under the decomposition map give a corresponding filtration of the Specht module $S_F^\beta=S_{\mathbb Z}^\beta/pS_{\mathbb Z}^\beta$. We note that
\begin{equation}\label{E:Jantzen}
J^i(S_{\mathbb Z}^\beta)=S_{\mathbb Z}^\beta\cap(p^iM_{\mathbb Z}^\beta+S_{\mathbb Z}^{\beta\perp}).
\end{equation}
See \cite[1.2]{KuenzerNebe} for example. We prove Theorem \ref{T:Jantzen} in the following form:

\begin{Theorem}\label{T:Jantzen_proof}
Let $\alpha,\beta$ be partitions of $n$, with $\beta_a=\alpha_a-1$ and $\beta_b=\alpha_b+1$, for $a<b$, and $\beta_j=\alpha_j$, for $j\ne a,b$. Suppose that $p^i|(\alpha_a-\beta_b+b-a)$, for some $i\geq0$. Then each $\hat\theta\in\operatorname{Hom}_{\Sigma_n}(S^\alpha,S^\beta)$ satisfies $\operatorname{Im}(\hat\theta)\subseteq J^i(S^\beta)$.
\begin{proof}
In view of Corollary \ref{C:explicit}, we may assume that $\hat\theta$ is the map of the previous section, defined over ${\mathbb Z}$. Recall that the polytabloid $e_t:=\{t\}C_t^-$ generates $S^\alpha$. Here $t$ is a fixed $\alpha$-tableau, and $C_t^-$ is the signed column stabilizer sum of $t$. Also $h_a:=(\alpha_a-\beta_b+b-a)$. The strategy of this proof is very simple. In \eqref{E:adjustment} below we produce an `error term' ${\mathcal E}_t\in M^\beta$ and show that
$$
e_t\hat{\theta}-h_a{\mathcal E}_t\in S^\beta_{\mathbb Z}.
$$
Now over ${\mathbb Z}$, the image of $\hat\theta$ is contained in $S_{\mathbb Z}^{\beta\perp}$. So by \eqref{E:Jantzen}, $e_t\widehat{\theta}-h_a{\mathcal E}_t$ belongs to $J^i(S^\beta_{\mathbb Z})$, which is what we need to prove.

If $T$ is a semistandard $\alpha$-tableau of type $\beta$ then there is a corresponding semi-standard homomorphism $\widehat{\theta_T}:S^\alpha\rightarrow M^\beta$ such that
$$
e_t\widehat{\theta_T}=\sum_{U\approx T}UC_t^-.
$$
By Lemma \ref{L:repeated}, we may assume that $U$ ranges over all tableaux that are row equivalent to $T$, such that no column of $U$ contains a repeated entry.

Suppose that $U,V$ are $\alpha$-tableaux. Write $U<V$ if $U\approx V$ and all nodes $(a,c)$ in $U$ with $c>\beta_{a+1}$ contain the symbol $a$. We say that $U$ is {\em right-justified}. Now define the following linear combination of $\alpha$-tableaux (which will belong to $M^\gamma$, where $\gamma$ is the type of $V$):
$$
\widetilde{\theta_V}:=\sum_{U<V}UC_t^-.
$$
Now recall that $T$ is a semi-standard $\alpha$-tableau of type $\beta$. The analogues of Lemmas \ref{L:psi{ir}}, \ref{L:inot}, \ref{L:i-1bad} and \ref{L:i-1good} hold for $\widetilde{\theta_T}$, with the following exception:
\begin{equation}\label{E:exception}
\widetilde{\theta_T}\psi_a=\widetilde{\theta_W},\quad\mbox{if $\hat{T}(a)=a+1$,}
\end{equation}
and not $(h_a-h_{a+1})\widetilde{\theta_W}$. This can be seen by inspecting the original proof. The key point is that there are $\alpha_a-\alpha_{a+1}+1$ $\alpha$-tableaux $U$ such that $U\approx T$ and $U\psi_a=W$. However, only one of these $U$ satisfies $U<T$.

Now define the following element of $M^\beta$:
\begin{equation}\label{E:adjustment}
{\mathcal E}_t:=\sum_{\hat{T}(a)=a+1}\bigg((-1)^{|T|}\hspace{-.6em}\prod_{u\in\{R\}\backslash\{T\}}\hspace{-.3em} h_u\bigg)\widetilde{\theta_T}.
\end{equation}
So $T$ runs over all semi-standard $\alpha$-tableaux of type $\beta$ which contain a (unique) symbol $a+1$ in row $a$. Lemma \ref{L:all_but_a} states that $\widehat{\theta}\psi_{i,r}=0$, unless $i=a,r=\beta_{a+1}-1$. In view of the previous paragraph, the analogous result holds for ${\mathcal E}_t$:
$$
{\mathcal E}_t\,\psi_{i,r}=0,\quad\mbox{unless $i=a$ and $r=\beta_{a+1}-1$}.
$$
In view of this, and \eqref{E:SpechtIntersection}, we see that $e_t\widehat{\theta}-h_a{\mathcal E}_t\in S^\beta$, if and only if $(e_t\widehat{\theta}-h_a{\mathcal E}_t)\psi_a=0$. But this equality holds, by Lemma \ref{L:the_a_case} and \eqref{E:exception}.
\end{proof}
\end{Theorem}

As an example, let $p=5$ and $\alpha=(4,3)$ and $\beta=(3^2,1)$. So $a=1,b=3$ and $h_1=5,h_2=3$. There are two semi-standard $\alpha$-tableaux of type $\beta$:
$$
\tiny{\young(1112,223)}\quad\mbox{\normalsize and}\tiny{\quad\young(1113,222)}.
$$
corresponding to the semi-standard $\alpha$-sets $\{1,2\}$ and $\{1\}$, respectively. Our linear combination $\theta:S_{\mathbb Z}^\alpha\rightarrow M_{\mathbb Z}^\beta$ is then
$$
\theta=\tiny{\widehat{\theta_{\young(1112,223)}}}-3\tiny{\widehat{\theta_{\young(1113,222)}}}.
$$
The error term is
$$
{\mathcal E}_t=(\tiny{\young(2111,322)}+\tiny{\young(1211,232)}+\tiny{\young(1121,223)})C_t^-.
$$
Now consider the following standard $\alpha$-tableau $t$ and its signed column sum $C_t^-$:
$$
t=\tiny\young(1234,567),\quad C_t^-=(1-(1,5))(1-(2,6))(1-(3,7)).
$$
We can write $e_t\theta-h_a{\mathcal E}_t$ as a linear combination of $\alpha$-tabloids:
$$
\begin{array}{lcl}
e_t\theta-5{\mathcal E}_t
&=&\left(\begin{array}{l}
\tiny{\young(1112,223)}+\tiny{\young(1112,232)}+\tiny{\young(1112,322)}+\tiny{\young(1121,223)}+\tiny{\young(1211,232)}+\tiny{\young(2111,322)}\\
-3\left(\tiny{\young(1113,222)}+\tiny{\young(1131,222)}+\tiny{\young(1311,222)}+\tiny{\young(3111,222)}\right)\\
-5\left(\tiny{\young(2111,322)}+\tiny{\young(1211,232)}+\tiny{\young(1121,223)}\right)\end{array}\right)C_t^-\\
&=&\left(\begin{array}{l}
\tiny{\young(1112,223)}+\tiny{\young(1131,222)}-\tiny{\young(1113,222)}\\
\tiny{\young(1112,232)}+\tiny{\young(1311,222)}-\tiny{\young(1113,222)}\\
\tiny{\young(1112,322)}+\tiny{\young(3111,222)}-\tiny{\young(1113,222)}
\end{array}\right)C_t^-.\\
&=&\left(\tiny{\young(1112,223)}+\tiny{\young(1131,222)}+\tiny{\young(1123,221)}\right)\,C_t^-\left(1+(2,3)(6,7)+(1,3)(5,7)\right).
\end{array}
$$
Now
$$
\left(\tiny{\young(1112,223)}+\tiny{\young(1131,222)}+\tiny{\young(1123,221)}\right)C_t^-=
\left(\tiny{\young(321,465,7)}+\tiny{\young(421,765,4)}+\tiny{\young(721,365,4)}\right)C_t^-=e_{\tiny{\young(321,465,7)}}.
$$
Multiplying through by $\left(1+(2,3)(6,7)+(1,3)(5,7)\right)$, we get
$$
e_t\theta-5{\mathcal E}_t=e_{\tiny{\young(321,465,7)}}+e_{\tiny{\young(231,475,6)}}+e_{\tiny{\young(123,467,5)}}.
$$
We wish to compare this with the expression given by \eqref{Ex:OneBox}. So we use the equalities
$$
\begin{array}{lcl}
e_{\tiny{\young(321,465,7)}}
		&=&e_{\tiny{\young(123,564,7)}}+e_{\tiny{\young(124,567,3)}}-e_{\tiny{\young(123,567,4)}}\\
e_{\tiny{\young(231,475,6)}}
		&=&e_{\tiny{\young(132,574,6)}}+e_{\tiny{\young(134,576,2)}}-e_{\tiny{\young(123,567,4)}}\\
e_{\tiny{\young(123,467,5)}}
		&=&-e_{\tiny{\young(123,567,4)}}.
\end{array}
$$
Thus
$$
e_t\theta-5{\mathcal E}_t=e_{\tiny{\young(123,564,7)}}+e_{\tiny{\young(124,567,3)}}+
             e_{\tiny{\young(132,574,6)}}+e_{\tiny{\young(134,576,2)}}-3\,e_{\tiny{\young(123,567,4)}}.
$$

\section{Acknowledgement}

Part of this paper was written while Dr. Ellers was visiting the National University of Ireland Maynooth in July/August 2004. The visit was substantially funded by a {\em New Researcher Award} from the National University of Ireland Maynooth. We gratefully acknowledge this assistance.


\end{document}